\documentclass[10pt]{amsart}
\usepackage{pifont}
\usepackage{amssymb,amsthm,amsmath}
\usepackage[numbers,sort&compress]{natbib}
\usepackage{color}
\usepackage{graphicx}
\usepackage{tikz}

\title[ ]{The M\"{o}bius transformation of  continued fractions  with  bounded upper and lower  partial quotients}

\author{Wencai Liu}
\address[Wencai Liu]{Department of Mathematics, University of California, Irvine, California 92697-3875, USA} \email{liuwencai1226@gmail.com}

\newcommand{\R}{\mathbb{R}}
\newcommand{\Z}{\mathbb{Z}}

\newcommand{\N}{\mathbb{N}}
\theoremstyle{plain}
\newtheorem{theorem}{Theorem}[section]

\newtheorem{lemma}[theorem]{Lemma}

\theoremstyle{definition}

\begin{document}


\begin{abstract}

Let $h$: $x\mapsto \frac{ax+b}{cx+d} $ be the nondegenerate   M\"{o}bius  transformation  with integer entries. We  get a  bound of the continued fraction of $h(x)$ by the upper and lower bound of continued fraction of $x$,
which extends a result of Stambul \cite{S2}.

\end{abstract}
\maketitle

 \section{Introduction}
 A continued fraction representation  of  a number $x\in \R$ is an expansion of the form
 \begin{equation}\label{CE}
    x=a_0+\frac{1}{a_1+\frac{1}{a_2+\frac{1}{a_3+\frac{1}{\ddots}}}}
 \end{equation}
 where $a_0\in \Z$ and $a_i\in \N^{+}$, $i=1,2,\cdots$.
 A continued fraction may be finite or infinite. If  (\ref{CE}) is
a finite  continued fraction,  we denote  it by $[a_0; a_1,a_2,\cdots,a_n]$; if  (\ref{CE}) is infinite, then we denote it by
$[a_0; a_1,a_2,\cdots]$.  We call $a_j$ the $j$th partial quotient. It is a well known fact that the continued fraction of $x$ is infinite iff
$x$ is irrational.

Given a nondegenerate  $2\times 2$  matrix $M$ with integer entries,   that is $M=\left(
                                 \begin{array}{cc}
                                   a & b \\
                                   c & d \\
                                 \end{array}
                               \right)
$, where $a,b,c,d\in \Z$ and the determinant $ad-bc\neq 0$, we can define the associated M\"{o}bius transformation
 $h$: $x\mapsto \frac{ax+b}{cx+d} $. We also denote by
\begin{equation*}
  h(x)=  \left(
                                 \begin{array}{cc}
                                   a & b \\
                                   c & d \\
                                 \end{array}
                               \right)\cdot x=\frac{ax+b}{cx+d}
\end{equation*}

In this paper, we study  the bound of partial quotients under the M\"{o}bius transformation.
We will use $\lfloor x\rfloor=\max\{j\in \Z: j\leq x\}$.
Our main result is
\begin{theorem}\label{maintheorem}
Let $M=\left(
                                 \begin{array}{cc}
                                   a & b \\
                                   c & d \\
                                 \end{array}
                               \right)
$ be a nondegenerate  matrix with entries in $\Z$  and $h$ be the associated M\"{o}bius transformation. Let $x=[a_0;a_1,a_2,\cdots]$ be a real number such that
$B_1\leq a_j\leq B_2$ for $j$ large enough. Let $h(x)=[a_0^{\star};a_1^{\star},a_2^{\star},\cdots]$. Then
$a_j^{\star}\leq \lfloor\frac{D-1}{B_1}\rfloor +\lfloor D \frac{B_1B_2+\sqrt{B_1^2B_2^2+4B_1B_2}}{2B_1}\rfloor$ for large  $j$,
where $D=|\text{det}(M)|$.
\end{theorem}


Now we always assume $x=[a_0;a_1,a_2,\cdots]$ and  $\frac{ax+b}{cx+d}=[a_0^{\star};a_1^{\star},a_2^{\star},\cdots]$ with
$D=|ad-bc|\geq 1$. Set $M=\left(
                                  \begin{array}{cc}
                                    a & b \\
                                    c & d \\
                                  \end{array}
                                \right)
$ and $h(x)=\frac{ax+b}{cx+d}$.

It is an old result that a real number $\frac{ax+b}{cx+d}$ has bounded partial quotients if $x$ does \cite{Ha,Sh1,Ra}, so
   the quantitative bound becomes an interesting question.
Lagarias-Shallit \cite{LaSh} and Cusick-France \cite{CuMe} obtained  a quantitative bound, which stated that if $x$
has bounded partial quotients with $a_j\leq K$ eventually, then the associated partial quotients $a_j^{\star}$ of $\frac{ax+b}{cx+d}$ satisfy
$a_j^{\star}\leq D(K+2)$ eventually.

Using an algorithm developed by Liardet-Stambul \cite{S1} to calculate the partial quotients of $h(x)$,
Stambul gave a    upper bound $a_j^{\star}\leq D-1 +\lfloor D \frac{K+\sqrt{K^2+4K}}{2}\rfloor$\cite{S2}, which is the  $B_1=1$ case of Theorem \ref{maintheorem}.
In this paper, we concern the partial quotients with lower and upper bound at the same time. Our methods are based
on the refining of  analysis in papers \cite{S1,S2}.


\section{Algorithm for partial quotients}
In this section, we will introduce some notations  and the  algorithm developed by Liardet-Stambul \cite{S1,S2} to calculate the partial quotients of $h(x)$.
Let $M_{2,\N}$ be the set of
all matrices $M=\left(
                  \begin{array}{cc}
                    a & b \\
                    c & d \\
                  \end{array}
                \right)
$ $(a,b,c,d\in \N)$ such that $ad-bc\neq 0$. $M$ is said to be
in $\mathcal{D}_2$ when $a\geq c$ and $b\geq d$, in  $\mathcal{D}_2^{\prime}$ when  $a\leq c$ and  $b\leq d$, and in $ \varepsilon_2$ when
$ (a-c)(b-d)<0$.  $\{\mathcal{D}_2,\mathcal{D}_2^{\prime},\varepsilon_2\}$ is a partition of $M_{2,\N}$.

It is easy to see that
  $M\in\varepsilon_2$ satisfies
\begin{equation}\label{euq4}
\max\{|a|+ |b|, |c|+|d|\} \leq |\text{det}M|=D.
\end{equation}
For all matrices $M \in \mathcal{D}_2\cup\mathcal{D}_2^{\prime}$, there exists a unique
factorization
\begin{equation}\label{Fac}
    M=\left(
        \begin{array}{cc}
          c_0 & 1 \\
          1 & 0 \\
        \end{array}
      \right)
      \left(
        \begin{array}{cc}
          c_1 & 1 \\
          1 & 0 \\
        \end{array}
      \right)
      \cdots
      \left(
        \begin{array}{cc}
           c_n& 1 \\
          1 & 0\\
        \end{array}
      \right)M^{\prime}
\end{equation}
such that $c_0\in \N$, $c_1,\cdots,c_n\in \N^{+}$ and $M^{\prime}\in \varepsilon_2$\cite{S1}.
 This factorization will be
denoted by $M=\Pi_{c_0c_1,\cdots,c_n}M^{\prime}$. Moreover, $[c_0;c_1,c_2,\cdots,c_{n-1}]$ is the common sequence of  partial
quotients of $\frac{a}{c}$ and $\frac{b}{d}$ if $n\neq 1$.
$c_n$ can de determined by the following several cases\cite{S1}.
\begin{itemize}
  \item [\textbf{Case} 1]: If  $\frac{a}{c}=[c_0;c_1,c_2,\cdots,c_{n-1}]$, then $c_n$ is the $n$th partial quotient of $\frac{b}{d}$.
  \item  [\textbf{Case} 2]: If  $\frac{b}{d}=[c_0;c_1,c_2,\cdots,c_{n-1}]$, then $c_n$ is the $n$th partial quotient of $\frac{a}{c}$.
  \item[\textbf{Case} 3]: Otherwise,     $c_n$ is the smaller one of   $n$th partial quotients of $  \frac{a}{c}$ and  $\frac{b}{d}$.
\end{itemize}

Assume $M\in \varepsilon_2$ and $h$ is the associated M\"{o}bius transformation. Let $x=[a_0;a_1,a_2,\cdots]>1$.  Recall the algorithm in \cite{S1,S2} to compute the partial quotients of $h(x)$.

\textbf{Step 0:} $M_0=M\in \varepsilon_2,j=0,n=0$.

Let $j_1$ be the smallest positive integer (see \cite{S1} for the existence) such that $M_0\Pi_{a_0a_1\cdots a_{j_1-1}}\in \varepsilon_2$
and $M_0\Pi_{a_0a_1\cdots a_{j_1}}\in \mathcal{D}_2\cup\mathcal{D}_2^{\prime}$.
Factorizing  $M_0\Pi_{a_0a_1\cdots a_{j_1}}$ as (\ref{Fac}), we get
\begin{equation}\label{Output0}\tag{Output-0}
   M_0\Pi_{a_0a_1\cdots a_{j_1}}=\left(
        \begin{array}{cc}
          c_0 & 1 \\
          1 & 0 \\
        \end{array}
      \right)
      \left(
        \begin{array}{cc}
          c_1 & 1 \\
          1 & 0 \\
        \end{array}
      \right)
      \cdots
      \left(
        \begin{array}{cc}
           c_{n_1}& 1 \\
          1 & 0\\
        \end{array}
      \right)M_1
\end{equation}
with $M_1\in \varepsilon_2$.

\textbf{Step 1:} $M_1\in \varepsilon_2,j=j_1+1,n=n_1+1$.

Let $j_2\geq j_1+1$ be the smallest positive integer such that $M_1\Pi_{a_{j_1+1}a_{j_1+2}\cdots a_{j_2-1}}\in \varepsilon_2$
and $M_1\Pi_{a_{j_1+1}a_{j_1+2}\cdots a_{j_2}}\in \mathcal{D}_2\cup\mathcal{D}_2^{\prime}$.
Factorizing  $M_1\Pi_{a_{j_1+1}a_{j_1+2}\cdots a_{j_2}}$ as (\ref{Fac}), we get
\begin{equation}\label{Output1}\tag{Output-1}
  M_1\Pi_{a_{j_1+1}a_{j_1+2}\cdots a_{j_2}}=\left(
        \begin{array}{cc}
          c_{n_1+1} & 1 \\
          1 & 0 \\
        \end{array}
      \right)
      \left(
        \begin{array}{cc}
          c_{n_1+2} & 1 \\
          1 & 0 \\
        \end{array}
      \right)
      \cdots
      \left(
        \begin{array}{cc}
           c_{n_2}& 1 \\
          1 & 0\\
        \end{array}
      \right)M_2
\end{equation}
with $M_2\in \varepsilon_2$.

\textbf{Step 2:} $M_2\in \varepsilon_2,j=j_2+1,n=n_2+1$.

Let $j_3\geq j_2+1$ be the smallest positive integer such that $M_2\Pi_{a_{j_2+1}a_{j_2+2}\cdots a_{j_3-1}}\in \varepsilon_2$
and $M_2\Pi_{a_{j_2+1}a_{j_2+2}\cdots a_{j_3}}\in \mathcal{D}_2\cup\mathcal{D}_2^{\prime}$.
Factorizing  $M_2\Pi_{a_{j_2+1}a_{j_2+2}\cdots a_{j_3}}$ as (\ref{Fac}), we get
\begin{equation}\label{Output2}\tag{Output-2}
  M_2\Pi_{a_{j_2+1}a_{j_2+2}\cdots a_{j_3}}=\left(
        \begin{array}{cc}
          c_{n_2+1} & 1 \\
          1 & 0 \\
        \end{array}
      \right)
      \left(
        \begin{array}{cc}
          c_{n_2+2} & 1 \\
          1 & 0 \\
        \end{array}
      \right)
      \cdots
      \left(
        \begin{array}{cc}
           c_{n_3}& 1 \\
          1 & 0\\
        \end{array}
      \right)M_3
\end{equation}
with $M_3\in \varepsilon_2$.

\begin{center}
$\cdots\cdots\cdots$
\end{center}
\begin{center}
$\cdots\cdots\cdots$
\end{center}

\textbf{Step k:} $M_k\in \varepsilon_2,j=j_k+1,n=n_k+1$.

Let $j_{k+1}\geq j_k+1$ be the smallest positive integer such that $M_k\Pi_{a_{j_k+1}a_{j_k+2}\cdots a_{j_{k+1}-1}}\in \varepsilon_2$
and $M_k\Pi_{a_{j_k+1}a_{j_k+2}\cdots a_{j_{k+1}}}\in \mathcal{D}_2\cup\mathcal{D}_2^{\prime}$.
Factorizing  $M_k\Pi_{a_{j_k+1}a_{j_k+2}\cdots a_{j_{k+1}}}$ as (\ref{Fac}), we get
\begin{equation}\label{Outputk}\tag{Output-k}
  M_k\Pi_{a_{j_k+1}a_{j_k+2}\cdots a_{j_{k+1}}}=\left(
        \begin{array}{cc}
          c_{n_k+1} & 1 \\
          1 & 0 \\
        \end{array}
      \right)
      \left(
        \begin{array}{cc}
          c_{n_k+2} & 1 \\
          1 & 0 \\
        \end{array}
      \right)
      \cdots
      \left(
        \begin{array}{cc}
           c_{n_{k+1}}& 1 \\
          1 & 0\\
        \end{array}
      \right)M_{k+1}
\end{equation}
with $M_{k+1}\in \varepsilon_2$.

Putting  all the Output (\ref{Outputk}) together, we get a sequence

\begin{equation}\label{Alloutput}\tag{Alloutput-k}
    c_0c_1c_2c_3\cdots c_{n_k}
\end{equation}
Unfortunately, many $c_i$ maybe zero, thus we must introduce the contraction map $\mu$.
For any word  $c_0c_1c_2c_3\cdots c_n\in \N^n$, let $\mu$ be the contraction map which transforms a word into a word
where all letters are positive integers (except perhaps the first one), replacing from
left to right factors $a0b$ by the letter $a + b$.

By the fact
\begin{equation*}
    \left(
       \begin{array}{cc}
         a & 1 \\
         1 & 0 \\
       \end{array}
     \right)
     \left(
       \begin{array}{cc}
         b & 1 \\
         1 & 0\\
       \end{array}
     \right)=
     \left(
       \begin{array}{cc}
         a+b & 1 \\
         1 & 0\\
       \end{array}
     \right),
\end{equation*}
we have
\begin{equation}\label{equ2}
   \Pi_{\mu(c_0c_1c_2c_3\cdots c_n)} =\Pi_{c_0c_1c_2c_3\cdots c_n}
\end{equation}

Let $\mu$ act on (\ref{Alloutput}), then we get
\begin{equation}\label{Alloutput1}\tag{Partialquotients}
     c_0^{\star}c_1^{\star}c_2^{\star}c_3^{\star}\cdots c^{\star}_{n_k^{\prime}}=\mu(c_0c_1c_2c_3\cdots c_{n_k}).
\end{equation}

By the arguments in \cite{S1}, $n_k^{\prime}$ goes to infinity as $k$ does, moreover,
\begin{equation}\label{equ3}
    \frac{ax+b}{cx+d}=[c_0^{\star};c_1^{\star},\cdots,c^{\star}_{n_k^{\prime}-1},\cdots]
\end{equation}
and the $n_k^{\prime}$th partial quotient following $c^{\star}_{n_k^{\prime}-1}$ is no less than $c^{\star}_{n_k^{\prime}}$.

Now, we give a quantitative estimate about  $c_i$ in (\ref{Alloutput}).
\begin{lemma}\label{Keyle}
Assume $M\in \varepsilon_2$ and $x=[a_0;a_1,a_2,\cdots]>1$. Let  $h$ be  the associated M\"{o}bius transformation and $D=|\text{det}M|\geq 1$.
Suppose $ a_j\leq K$ for some $K\in \N^{+}$. We do the algorithm as  above,
then the  following three claims hold,
\begin{description}
  \item[(i)] For any $n_k< j\leq n_{k+1}-1$, $c_j\leq D-1$
  \item[(ii)] For any $k$, $c_{n_{k+1}}\leq D K$
  \item[(iii)] If for some $k$, $c_{n_{k+1}}\geq D $, then the right upper entry of $M_{k+1}$ must be zero, that is  $M_{k+1}$
has the form
\begin{equation}\label{equ5}
    M_{k+1}=\left(
              \begin{array}{cc}
                \star & 0 \\
                \star & \star \\
              \end{array}
            \right)
\end{equation}
\end{description}

\end{lemma}
\begin{proof}
The three claims are from \cite{S2}. We rewrite the proof here to make the paper more readable.
By the algorithm, we already have
 $M_k\Pi_{a_{j_k+1}a_{j_k+2}\cdots a_{j_{k+1}-1}}\in \varepsilon_2$
and $M_k\Pi_{a_{j_k+1}a_{j_k+2}\cdots a_{j_{k+1}}}\in \mathcal{D}_2\cup\mathcal{D}_2^{\prime}$.

For simplicity, let  $M^\prime=\left(
                          \begin{array}{cc}
                            \alpha & \beta\\
                            \gamma & \delta \\
                          \end{array}
                        \right)=
M_k\Pi_{a_{j_k+1}a_{j_k+2}\cdots a_{j_{k+1}-1}}\in \varepsilon_2$ and $f=a_{j_{k+1}}\leq K$. Then
$M^\prime\Pi_f\in \mathcal{D}_2\cup\mathcal{D}_2^{\prime}$.

If $\gamma=0$, then
\begin{equation*}
    M^\prime\Pi_f=\left(
                    \begin{array}{cc}
                      \alpha f+\beta & \alpha \\
                      \delta & 0 \\
                    \end{array}
                  \right)\in \mathcal{D}_2\cup\mathcal{D}_2^{\prime}
\end{equation*}
and we must have $\alpha f+\beta\geq \delta$.
Thus
\begin{equation*}
    M^\prime\Pi_f=\left(
                    \begin{array}{cc}
                      \alpha f+\beta & \alpha \\
                      \delta & 0 \\
                    \end{array}
                  \right)=
\left(
                    \begin{array}{cc}
                      \lfloor\frac{\alpha f+\beta}{\delta}\rfloor & 1 \\
                      1 & 0 \\
                    \end{array}
                  \right)\left(
                    \begin{array}{cc}
                     \delta  & 0 \\
                      (\alpha f+\beta)\text{ mod } \delta & \alpha\\
                    \end{array}
                  \right).
\end{equation*}
In this case, in order to prove the Lemma, it suffices to show that
\begin{equation}\label{equ6}
    \lfloor\frac{\alpha f+\beta}{\delta}\rfloor\leq  DK.
\end{equation}

Otherwise, one has
\begin{equation}\label{equ71}
   DK +1\leq \lfloor\frac{\alpha f+\beta}{\delta}\rfloor=  \lfloor\frac{\alpha f}{\delta}+\frac{\beta}{\delta}\rfloor \leq \lfloor\frac{\alpha K}{\delta}+\frac{\beta}{\delta}\rfloor,
\end{equation}
since $f\leq K$.

By the fact $M^\prime=\left(
                          \begin{array}{cc}
                            \alpha & \beta\\
                            0 & \delta \\
                          \end{array}
                        \right)\in \varepsilon_2$, we have $\beta< \delta$, $|\alpha|+|\beta|\leq D$. This is contradicted to (\ref{equ71}).

If $\alpha=0$, then
\begin{equation*}
    M^\prime\Pi_f=\left(
                    \begin{array}{cc}
                      \beta & 0 \\
                      \gamma f+\delta & \gamma \\
                    \end{array}
                  \right)\in \mathcal{D}_2\cup\mathcal{D}_2^{\prime}
\end{equation*}
and we must have $\gamma f+\delta\geq \beta$.
Thus
\begin{equation*}
    M^\prime\Pi_f=\left(
                    \begin{array}{cc}
                      0 & 1 \\
                      1 & 0 \\
                    \end{array}
                  \right)
\left(
                    \begin{array}{cc}
                      \lfloor\frac{\gamma f+\delta}{\beta}\rfloor & 1 \\
                      1 & 0 \\
                    \end{array}
                  \right)\left(
                    \begin{array}{cc}
                     \beta  & 0 \\
                      (\gamma f+\delta)\text{ mod } \delta & \gamma\\
                    \end{array}
                  \right).
\end{equation*}
In this case, we can still prove the Lemma like the case $\gamma=0$.

If $\alpha, \gamma\geq 1$, then
\begin{equation*}
    M^\prime\Pi_f=\left(
                    \begin{array}{cc}
                      \alpha f+\beta & \alpha \\
                      \gamma f+\delta & \gamma \\
                    \end{array}
                  \right)\in \mathcal{D}_2\cup\mathcal{D}_2^{\prime}.
\end{equation*}
By the  algorithm, $n_k\leq j\leq n_{k+1}-1$, $c_j$ is the common partial quotient of $ \frac{\alpha}{\gamma}$ and $\frac{\alpha f+\beta}{\gamma f+\delta}$.

 We first show claim 1 holds.  Indeed,   $ \alpha\leq D$ and $\gamma\geq 1$.
If $\alpha=D$ and $\gamma=1$, we must have $\beta=0$ and $\delta=1$. This implies claim 1 when we consider the partial quotient of $\frac{\alpha f+\beta}{\gamma f+\delta}$.
Otherwise ($\alpha=D$ and $\gamma=1$ do not hold)  claim 1 holds if we consider the partial quotient of $ \frac{\alpha}{\gamma}$.

Suppose the last letter, i.e. $c_{n_{k+1}}\geq D$, then we must have  $\frac{a}{c}=[c_{j_k+1};c_{j_k+2},c_{j_k+2},\cdots,c_{j_{k+1}-1}]$ by the (Case1-Case3) and
$c_{n_{k+1}}\geq D$ is the
$n_{k+1}-n_{k}+1$th partial quotient of $\frac{\alpha f+\beta}{\gamma f+\delta}$.
This implies claims 2 and 3 if we can show
\begin{equation*}
   \frac{1}{DK}\leq \frac{\alpha f+\beta}{\gamma f+\delta}\leq  DK.
\end{equation*}
We only prove the fact  $\frac{\alpha f+\beta}{\gamma f+\delta}\leq  DK$, the proof of lower bound $\frac{1}{DK}\leq \frac{\alpha f+\beta}{\gamma f+\delta}$ is the same.

If $\gamma f+\delta\geq 2$, then $   \frac{\alpha f+\beta}{\gamma f+\delta}\leq \frac{DK+D}{2}\leq DK$.
If $\gamma f+\delta\leq 1$, then we have $ \delta=0$ and $\gamma=K=1$. This implies $\beta=D$ and $\alpha=0$. We still have
$\frac{\alpha f+\beta}{\gamma f+\delta}\leq  DK$.

\end{proof}
\section{ Some Lemmas}
We say a M\"{o}bius transformation $h(\cdot)=M\cdot $ can not change the continued fraction  eventually, if for any $x$,
there exists some $N\in\N$ such that the $n$th partial quotients of $h(x)$ and $x$ are the same for any $n\geq N$.
\begin{lemma}\label{keyle2}
The following forms of M\"{o}bius transformations can not change the continued fraction  eventually,
\begin{equation}\label{mar}
   S=\{ \left(
      \begin{array}{cc}
        1 & k_1 \\
        0 & 1 \\
      \end{array}
    \right),
     \left(
      \begin{array}{cc}
        k_2 & 1 \\
        1 & 0 \\
      \end{array}
    \right), \left(
      \begin{array}{cc}
        1 & 0 \\
        k_3 & 1 \\
      \end{array}
    \right),
    \left(
      \begin{array}{cc}
        -1 & 0 \\
        0 & 1 \\
      \end{array}
    \right),
    \left(
      \begin{array}{cc}
        1 & 0 \\
        0 & -1 \\
      \end{array}
    \right),
    \left(
      \begin{array}{cc}
        0 & 1 \\
        1 & 0 \\
      \end{array}
    \right)\},
\end{equation}
where $k_1,k_2,k_3\in\Z$.
\end{lemma}
\begin{proof}
The proof is based on direct computation.
\end{proof}
{\bf Remark}{: The determinant of each matrix in $S$ is $\pm 1$.}
\begin{lemma}\label{keyle3}
Assume $a,b,c,d\in \Z$ and $ad-bc\neq 0$, then $M=\left(
                                                \begin{array}{cc}
                                                  a & b \\
                                                  c & d \\
                                                \end{array}
                                              \right)
$ can be rewritten in the following form
\begin{equation}\label{form}
    M=S_1S_2\cdots S_n M^{\prime}
\end{equation}
with $M^{\prime}\in \varepsilon_2$. Moreover if $D=\text{det} M=1$, then $M^{\prime}$ can be $\left(
                                                                                   \begin{array}{cc}
                                                                                     1 & 0 \\
                                                                                     0 & 1 \\
                                                                                   \end{array}
                                                                                 \right)
$.
\end{lemma}
\begin{proof}
Using M\"{o}bius transformation $\left(
         \begin{array}{cc}
           -1 & 0 \\
           0 & 1 \\
         \end{array}
       \right)
\in S$ and $\left(
        \begin{array}{cc}
          1 & 0\\
          0 & -1 \\
        \end{array}
      \right)\in S$, we can assume $a,c\geq 0$.

      Using  M\"{o}bius transformation $\left(
         \begin{array}{cc}
           1 & 0 \\
           k & 1 \\
         \end{array}
       \right)
\in S$ and $\left(
        \begin{array}{cc}
          0 & 1\\
          1 & 0 \\
        \end{array}
      \right)\in S$, $M$ can be  changed to $M_1=\left(
                                                   \begin{array}{cc}
                                                     a_1 & b_1 \\
                                                     0 & d_1\\
                                                   \end{array}
                                                 \right)
      $ with $a_1\geq 1.$

      Using  M\"{o}bius transformation $\left(
         \begin{array}{cc}
           1 & 0 \\
            & -1 \\
         \end{array}
       \right)
\in S$ and $\left(
        \begin{array}{cc}
          1& k\\
          0 & 1 \\
        \end{array}
      \right)\in S$, $M_1$ can be  changed to $M^{\prime}=\left(
                                                   \begin{array}{cc}
                                                     a_1 & b_1\text{ mod }|d_1| \\
                                                     0 & |d_1|\\
                                                   \end{array}
                                                 \right)\in \varepsilon_2
      $.

Moreover, if $D=1$, we must have $a_1=1,|b_1|=1$ and $b_1\text{ mod }|d_1| =0$.

\end{proof}

{\bf Remark}: If $|\text{det}M|=1$, then  the associated  M\"{o}bius transformations can not change the continued fraction  eventually.

\begin{lemma}\label{keyle6}
Let $M\in \varepsilon_2$ and $D=|\text{det} M|\geq 2$. Let $x=[a_0;a_1,a_2,\cdots]$ such that $B_1\leq a_j\leq B_2$ for all $j\geq 0$.
Using the Algorithm in section 2, we get a sequence  $  c_0^{\star}c_1^{\star}c_2^{\star}c_3^{\star}\cdots$ by (\ref{Alloutput1}).
If $c_0^{\star}=0$, then
\begin{equation}\label{equ91}
 c_1^{\star}\leq   \lfloor D y_0\rfloor
\end{equation}
where $y_0=[B_2;B_1,B_2,B_1,\cdots]\triangleq[\overline{B_2,B_1}]=\frac{B_1B_2+\sqrt{B_1^2B_2^2+4B_1B_2}}{2B_1}$.
Moreover, the equality in (\ref{equ91}) holds iff $a=0,b=1,c=D$ and $d=0$.

In addition,  assume  $M\neq \left(
                                                                                                        \begin{array}{cc}
                                                                                                          0 & 1 \\
                                                                                                          D & 0 \\
                                                                                                        \end{array}
                                                                                                      \right)
$, then
\begin{equation}\label{equ10}
 c_1^{\star}\leq  \max\{\lfloor \frac{D}{4}y_0+1\rfloor, D-1\}
\end{equation}
if $c_0^{\star}=0$.
\end{lemma}
\begin{proof}
Let
\begin{equation*}
    \frac{p_{n}}{q_n}=[a_0;a_1,a_2,\cdots,c_n],
\end{equation*}
then
\begin{equation*}
    \Pi_{a_0a_1\cdots a_n}=\left(
                             \begin{array}{cc}
                               p_n & p_{n-1} \\
                               q_n & q_{n-1} \\
                             \end{array}
                           \right).
\end{equation*}
Thus we have the following simple facts
\begin{equation}\label{equ10}
   M \Pi_{a_0a_1\cdots a_n}=\left(
                              \begin{array}{cc}
                                ap_n+bq_{n} &  ap_{n-1}+bq_{n-1} \\
                                cp_n+dq_{n} & cp_{n-1}+dq_{n-1} \\
                              \end{array}
                            \right),
\end{equation}
and
\begin{equation*}
     \lim_{n\to \infty}\frac{ap_n+bq_{n}}{cp_n+dq_{n}}=\frac{ap_{n-1}+bq_{n-1}}{cp_{n-1}+dq_{n-1}}=\frac{ax+b}{cx+d}.
\end{equation*}

If $ c_0^{\star}=0$, then $ c_1^{\star}$ is the second common partial quotient of $\frac{ap_n+bq_{n}}{cp_n+dq_{n}} $ and $\frac{ap_{n-1}+bq_{n-1}}{cp_{n-1}+dq_{n-1}}$ for any large $n$. Combining with (\ref{equ10}),
we must have
\begin{equation}\label{equ11}
    c_1^{\star}=\lfloor\frac{cx+d}{ax+b}\rfloor.
\end{equation}
Now we are in a position to prove the Lemma, based on (\ref{equ11}).

Case 1: $a\geq  1$

Using $x> 1$, one has
\begin{eqnarray*}
  \frac{cx+d}{ax+b} &\leq & \frac{cx+d}{ax} \\
   &< & \frac{c+d}{a} \\
   &\leq& D
\end{eqnarray*}
where the third inequality holds by (\ref{euq4}). This implies $ c_1^{\star}\leq D-1$.

Case 2: $a= 0$

In this case, we have
 $b>d$, $bc=D$ and $c+d\leq D$ by $M\in \varepsilon_2$, and
 \begin{equation}\label{equ12}
    c_1^{\star}=\lfloor\frac{D}{b^2}x+\frac{d}{b}\rfloor.
\end{equation}

If $ b\geq 2$,  by (\ref{equ12}), one has
\begin{equation*}
     c_1^{\star}\leq \lfloor\frac{D}{4}x+1\rfloor.
\end{equation*}
Notice that if a real number   with bounded partial quotients  in $[B_1,B_2]\cap \Z$
is such that $x\leq y_0$, then
\begin{equation*}
     c_1^{\star}\leq \lfloor\frac{D}{4}y_0+1\rfloor \leq \lfloor Dy_0\rfloor-1,
\end{equation*}
since $y_0\geq \frac{\sqrt{5}+1}{2}$ and $D\geq 2$.

If $b=1$, we must have $c=D$ and $d=0$.

Putting all the cases together, we complete the proof.
\end{proof}

\begin{lemma}\label{keyle7}
Let $M\in \varepsilon_2$ with the form $ \left(
                                         \begin{array}{cc}
                                           a & 0 \\
                                           c & d \\
                                         \end{array}
                                       \right)
$ and $D=|\text{det} M|\geq 1$. Let $x=[a_0;a_1,a_2,\cdots]$ such that $B_1\leq a_j\leq B_2$ for all $j\geq 0$.
Applying the Algorithm  in section 2 to $M\cdot x$, we get a sequence  $  c_0^{\star}c_1^{\star}c_2^{\star}c_3^{\star}\cdots$ by (\ref{Alloutput1}).
If $c_0^{\star}=0$, we must have
\begin{equation*}
 c_1^{\star}\leq    \lfloor\frac{D}{x_0}\rfloor,
\end{equation*}
where $x_0=[B_1;B_2,B_1,B_2,\cdots]\triangleq[\overline{B_1,B_2}]=\frac{B_2B_1+\sqrt{B_1^2B_2^2+4B_1B_2}}{2B_2}$.
\end{lemma}
\begin{proof}
Let $b=0$ in (\ref{equ11}), then we get
\begin{equation}\label{equ13}
    c_1^{\star}=\lfloor\frac{cx+d}{ax}\rfloor.
\end{equation}
Notice that if  a real number   with bounded partial quotients  in $[B_1,B_2]\cap \Z$
is such that $x\geq  x_0$, then
\begin{equation}\label{equ13}
    c_1^{\star}\leq \lfloor\frac{cx_0+d}{ax_0}\rfloor.
\end{equation}
Thus in order to prove this Lemma,  it suffices to
show
\begin{equation}\label{equ14}
   \frac{cx_0+d}{ax_0}\leq \frac{D}{x_0}.
\end{equation}.

If $a=1$, we must have $c=0$ and $d=D$,
this implies (\ref{equ14}).

If $a\geq 2$, we already have $ ad=D$ and $c\leq a-1$.

Case 1: $D\geq 2x_0>2$

One has
\begin{eqnarray*}
  c x_0+d &\leq & (a-1)x_0+\frac{D}{2} \\
   &\leq &  \frac{D(a-1)}{2}+\frac{D}{2} \\
  &\leq & Da
\end{eqnarray*}
This implies (\ref{equ14}).

Case 2: $x_0\leq D<2 x_0$

 It suffices to
show
\begin{equation}\label{equ15}
   \frac{cx_0+d}{ax_0}<2.
\end{equation}

This is obvious by the following computation,\begin{eqnarray*}
                                            c x_0+d &\leq & (a-1)x_0+ D \\
   &< &  a x_0+2x_0\\
  &\leq  &  2ax_0
                                             \end{eqnarray*}
This implies (\ref{equ15}).

Case 3: $D<x_0$

By direct computation,
\begin{eqnarray*}
  \frac{cx_0+d}{ax_0} &=& \frac{c}{a}+\frac{D}{a^2x_0} \\
   &< & \frac{a-1}{a}+\frac{1}{a^2}\\
   &< & 1 .
\end{eqnarray*}
This also implies (\ref{equ14}).

\end{proof}
\section{Proof of  Theorem \ref{maintheorem}}
{\bf{Proof  of Theorem \ref{maintheorem}}}:
\begin{proof}
Suppose $x=[a_0;a_1,a_2,\cdots]$  is such that
$B_1\leq a_j\leq B_2$ for $j\geq j_0$, and $M=\left(
                                           \begin{array}{cc}
                                             a & b \\
                                             c & d \\
                                           \end{array}
                                         \right)
$ is such that $D=|\text{det}M|\geq 1$.
By Lemmas \ref{keyle2} and \ref{keyle3}, we may assume $M\in \varepsilon_2$.
By the fact
\begin{equation}\label{equ7}
    h(x)=M\cdot x=M\Pi_{a_0a_1\cdots a_{j_0}}\cdot[a_{j_0+1};a_{j_0+2},\cdots]
\end{equation}
combining with (\ref{Fac}), in order to prove Theorem \ref{maintheorem}, we only need to prove the case
  when  all the partial quotients of $x$ satisfy $B_1\leq a_i\leq B_2$.

By the Algorithm, it suffices to show that
for any word $k_10k_20\cdots 0k_p$ in (\ref{Alloutput}) with $k_i\in\N^+, i=1,2,\cdots,p$, we have
\begin{equation}\label{equ9}
    k_1+k_2+\cdots+k_p\leq \lfloor\frac{D-1}{B_1}\rfloor +\lfloor D \frac{B_1B_2+\sqrt{B_1^2B_2^2+4B_1B_2}}{2B_1}\rfloor.
\end{equation}

Assume $k_1$ is the last letter of $k$th  step (\ref{Alloutput}). Then the output of $k+1$th step is $0k_2$, $k+2$th step is $0k_3$, $\cdots$.

\textbf{Case 1:} $ k_1\geq D$

By  (iii) of Lemma \ref{Keyle}, $M_{k+1}$ has the form
\begin{equation*}
    M_{k+1}=\left(
              \begin{array}{cc}
                a_{k} & 0 \\
                c_k & d_k \\
              \end{array}
            \right)\in \varepsilon_2.
\end{equation*}

By Lemma \ref{keyle7}, we have
\begin{equation*}
    \sum_{j=2}^pk_j\leq \lfloor\frac{D}{x_0}\rfloor.
\end{equation*}
By (ii) of Lemma \ref{Keyle}, $k_1\leq DB_2$,
then
\begin{eqnarray*}
  \sum_{j=1}^pk_j &\leq& \lfloor\frac{D}{x_0}\rfloor+DB_2 \\
   &\leq&  \lfloor D \frac{B_2B_1+\sqrt{B_1^2B_2^2+4B_1B_2}}{2B_1}\rfloor\\
 &\leq&   \lfloor\frac{D-1}{B_1}\rfloor +\lfloor D \frac{B_1B_2+\sqrt{B_1^2B_2^2+4B_1B_2}}{2B_1}\rfloor.
\end{eqnarray*}
This implies the Theorem in this case.

By the Remark following Lemma \ref{keyle3}, we can assume $D\geq 2$.

\textbf{Case 2:} $k_1\leq D-1$

 If  $M_{k+1}\neq \left(
                    \begin{array}{cc}
                      0 & 1 \\
                      D & 0 \\
                    \end{array}
                  \right)
  $,
  by (\ref{equ10}) one has
  \begin{equation*}
   \sum_{j=2}^pk_j \leq   \max\{\lfloor \frac{D}{4}y_0+1\rfloor, D-1\}.
  \end{equation*}
  Direct computation (spliting the computation into $B_1=1$ or $B_1\geq 2$),
  \begin{eqnarray*}
     \sum_{j=1}^pk_j  &\leq & D-1 +\max\{\lfloor \frac{D}{4}y_0+1\rfloor, D-1\}\\
 &\leq&   \lfloor\frac{D-1}{B_1}\rfloor +\lfloor D \frac{B_1B_2+\sqrt{B_1^2B_2^2+4B_1B_2}}{2B_1}\rfloor.
\end{eqnarray*}
This implies the Theorem in this case.

If  $M_{k+1}= \left(
                    \begin{array}{cc}
                      0 & 1 \\
                      D & 0 \\
                    \end{array}
                  \right)
  $, by (\ref{equ9}) one has
  \begin{equation*}
 c_1^{\star}\leq   \lfloor D y_0\rfloor.
\end{equation*}
Thus  in order to prove the Theorem in this case, it suffices to
show
\begin{equation}\label{equ16}
    k_1\leq\frac{D-1}{B_1}.
\end{equation}

  By the Algorithm of $k$th step, we have
  \begin{equation}\label{equ15}
    M_k\Pi_{a_1a_2\cdots a_N}=\Pi_{c_1c_2\cdots c_{N^{\prime}-1}} \left(
                                       \begin{array}{cc}
                                         k_1 & 1 \\
                                         1 & 0 \\
                                       \end{array}
                                     \right)\left(
                                              \begin{array}{cc}
                                                0 & 1 \\
                                               D & 0 \\
                                              \end{array}
                                            \right)\in  \mathcal{D}_2\cup\mathcal{D}_2^{\prime},
  \end{equation}
 and $M_k\Pi_{a_1a_2\cdots a_{N-1}}\in \varepsilon_2$.

 This implies
 \begin{equation}\label{equ16}
    M_k\Pi_{a_1a_2\cdots a_{N-1}}=\Pi_{c_1c_2\cdots c_{N^{\prime}-1}} \left(
                                       \begin{array}{cc}
                                         k_1 & 1 \\
                                         1 & 0 \\
                                       \end{array}
                                     \right)\left(
                                              \begin{array}{cc}
                                                0 & 1 \\
                                               D & 0 \\
                                              \end{array}
                                            \right)\left(
                                              \begin{array}{cc}
                                                a_N & 1 \\
                                               1 & 0 \\
                                              \end{array}
                                            \right)^{-1}.
  \end{equation}
  By direct computation, one has
  \begin{equation}\label{equ16}
    M_k\Pi_{a_1a_2\cdots a_{N-1}}=\Pi_{c_1c_2\cdots c_{N^{\prime}-1}} \left(
                                       \begin{array}{cc}
                                         k_1 & -k_1 a_N+D \\
                                         1 & -a_N \\
                                       \end{array}
                                     \right).
  \end{equation}

  Since all entries of $ M_k\Pi_{a_1a_2\cdots a_{N-1}}$ are non-negative, we must have
  \begin{equation}\label{equ17}
    -k_1 a_N+D\geq 1.
  \end{equation}
  This implies
  \begin{equation*}
    k_1\leq \lfloor \frac{D-1}{B_1} \rfloor,
  \end{equation*}
  since $a_N\geq B_1$. We complete the proof.

\end{proof}
 \section*{Acknowledgments}
 I would like to thank Svetlana
Jitomirskaya for comments   on  earlier versions of
the manuscript.
This research was partially
 supported   by the  AMS-Simons Travel Grant (2016-2018)  and   NSF DMS-1401204.

\end{document}